\documentclass[a4paper, 10pt]{article}
\usepackage{graphicx}

\usepackage{alltt}

\usepackage{amsmath}

\usepackage{chemformula} 
\usepackage[T1]{fontenc} 
\usepackage{latexsym}
\usepackage{amsmath}
\usepackage{amssymb}
\usepackage{amsthm}
\usepackage{color}
\usepackage{graphicx}
\usepackage{epstopdf}
\usepackage[normalem]{ulem}
\usepackage{float}
\usepackage{hyperref}

\newcommand{\erre}{\mathbb{R}}

\newtheorem{theorem}{Theorem}{}
{}
{}

\newtheorem{proposition}[theorem]{Proposition}
\newtheorem{definition}{Definition}

\usepackage{authblk}
\title{Fractional Lotka-Volterra model with time-delay and delayed controller for a bioreactor}
\author[1]{R. Villafuerte-Segura}
\author[2]{B.~A.~Itz\'a-Ortiz}
\author[3]{P.~A.~L\'opez-Perez}
\author[1]{E. Alvarado-Santos}
\affil[1]{Centro de Investigaci\'on en
Tecnolog\'ias de Informaci\'on y Sistemas, Universidad Aut\'onoma del Estado de Hidalgo, Pachuca, Hidalgo, Mexico, 42180}
\affil[2]{Centro de Investigaci\'on en Matem\'aticas,Universidad Aut\'onoma del Estado de Hidalgo, Pachuca, Hidalgo, Mexico, 42180 }
\affil[3]{Escuela Superior de Apan, Universidad Aut\'onoma del Estado de Hidalgo, Apan, Hidalgo, Mexico, 43900}
\date{}                     
\begin{document}
\maketitle

\begin{abstract}
In this paper, a fractional Lotka-Volterra mathematical model for a bioreactor is proposed and used to fit the data provided by  a bioprocess known as continuous fermentation of \textit{Zymomonas mobilis}. 
The model contemplates  a time-delay $\tau$   due to the dead-time in obtaining the measurement of biomass $x(t)$. A Hopf bifurcation analysis is performed 
to characterize the inherent self oscillatory experimental bioprocess response. As consequence, 
  stability conditions for the equilibrium point together with conditions for limit cycles using the delay $\tau$ as bifurcation parameter are obtained.
Under the assumptions  that the use of observers, estimators or extra laboratory measurements are avoided to prevent the rise of computational or monetary costs, for the purpose of control, we will only consider  the measurement of the biomass. A  simple controller that can be employed is the proportional action controller $u(t)=k_px(t)$, which is shown to fail  to stabilize the obtained model under the proposed analysis. Another suitable choice is the use of a delayed controller $u(t)=k_rx(t-h)$ which successfully stabilizes the model even when it is unstable.
 Finally, the proposed theoretical results are corroborated through numerical simulations. \vskip 2mm


\end{abstract}

\section{Introduction}\label{intro}

In addition  of its academic  and performance predictive importance, ma\-the\-ma\-ti\-cal models for bioreactors are required tools  for implementing process control strategies  demanded in industrial bioprocess. 
Using these devices, researches have successfully cultivated important kind of cells for the pharmaceutic industry \cite{Ho_etal}, developed complex organ cultures \cite{Gantenbein_etal,Schmidt}, designed wastewater treatments \cite{Appels,Price}, increased hydrogen production \cite{Lopez}, and studied microbial biomass and energy conversion \cite{Qiang}, among many other relevant applications.

In  the case of culturable microbial species, the literature contains numerous examples modeling their iteration \cite{MAHADEVAN2012e201210008,WADE201664,Wintermute}
and their consequent production of biochemical substances \cite{Pham}. 
A fundamental part of bioprocessing is to create  favorable environmental conditions for microorganisms based on different substrate compositions. In order to achieve this conditions, positive feedback loops are essential. Mathematically speaking, this amounts to the need to have parametric conditions for the characterization of bifurcations and the appearance of limit cycles of the dynamical system modeling the bioreactor. In the literature, modeling and bifurcation analysis using the dilution rate as a parameter are predominant.


It is known that laboratory observations exhibit oscillatory behavior. The oscillatory behavior manifests itself as the self-sustained oscillations (SSO) among key process variables and parameters (e.g.\ substrate, biomass, metabolites, dilution rate, temperature, agitation speed, pH, constant feed and culture conditions) \cite{2,1}. Consequently, diverse phenomena, such as  multiplicity of steady states, stability of limit cycles or chaoticity,  may be  present in the bioprocess \cite{4,3,5}.
In some bioprocesses it has been observed that the oscillatory mode of operation can have higher performance in comparison to the obtained in steady-state regime (e.g. higher  biomass and metabolites concentrations) \cite{7,8,6}. 
Therefore it as important to understand the conditions for the oscillatory process and as it is highly desirable to know how to attenuate the oscillatory changes in the product concentration (biomass and metabolites). The above can be resolved via the manipulation of a parameter (set point: dilution rate, temperature, pH or dissolved oxygen), the use of input effluent stabilization tanks, or by the design of  controllers that stabilize self-oscillating bioreactors by manipulating the input \cite{9}.

Naturally, delays are always present in bioprocesses, since there are dead times between biological reactions/interactions, metabolization of substances, considerations of past population rates, among others. Therefore, it is necessary to consider these in mathematical modeling \cite{gopalsamy2013stability,niculescu2007geometric,rubio2012control,villafuerte2007estimate,villafuerte2010stability,volterra1928theorie}. 

Some members of the scientific community have argued that time delays should be omitted or disregarded from mathematical models since they may cause undesirable/poor behavior in the system response or even compromise its stability.  However, this interpretation is wrong. In the last two decades a part of the scientific community has devoted efforts to studying the effect of considering time-delays,  dead-time or just delays  in  mathematical models  \cite{abdallah1993delayed,de1996controlling,suh1980use,suh1979proportional,swisher1988design,ulsoy2015time,villafuerte2013tuning,villafuerte2018tuning}. As previously mentioned, the delays are inherent (perceptible or not with the naked eye) in the real/experimental processes so that the omission of these in the mathematical models seems to be only for the purpose of facilitating its analysis.

Although the analysis of a mathematical model is rather harder for dynamical systems admitting a  delay,  this paper attempts to provide one more  example that this extra work is worthwhile as the use of a  delay allows the fitting of data from a real/experimental process. For example, delays seen as bifurcation parameters are useful tools in obtaining conditions for the existence of  limit cycles and Hopf bifurcations. Furthermore, delays may create  the conditions for the stabilization of a system otherwise unstable. In the last decade it has been shown that the deliberate inclusion of delays in the control laws can substitute the development of observers/estimators of unavailable state variables  \cite{kokame2002stability,niculescu2003some,olgac2002exact}. In addition to reducing the noise present in measuring instruments \cite{mondie2011tuning}. In bioprocesses, costs can be reduced by avoiding biomass or substrate measurement tests.

In this work  a fractional Lotka-Volterra model with time delay and delayed
controller for a bioreactor,  which will turn out to fit the data provided by  a bioprocess known as continuous fermentation of \textit{Zymomonas mobilis}, is proposed. Afterwards a mathematical analysis is conducted providing conditions for the  appearance of a Hopf bifurcations and stability of a equilibrium point using the delay as a bifurcation parameter. The novelty of our approach relays on three facts. The first one is that the  proposed mathematical model uses fractional exponents to scale the state variables, an approach successfully used to model the dynamics of cancer cell \cite{TannenbaumEtAl} and in epidemiology models \cite{Tagh}. Secondly is the
inclusion of a time-delay and thirdly the inclusion of  a delayed controller  to manipulate the inflow of the substrate to stabilize the bioreactor model even under unstable conditions.
This approach allows the design and tuning of delayed control laws to minimize costs and maximize productions of bioprocessing. To illustrate the theoretical results proposed here, we present numerical simulations.

The paper is organized as follows. The description of the biological system is shown in Section 2.  A description of the proposed fractional Lotka-Volterra  model with time-delay and delayed controller together with the corresponding mathematical analysis 
are presented in Section 3.  Here, we use $u(t)=D$ to identify the parameters of the model  and to characterize the critical parameters of $\tau$ which allow limit cycles and Hopf bifurcations. Then, we propose $u(t)=k_r x(t-h)$ for the design and tuning of a delayed controlled to $\sigma$-stabilize the model. In Section 4, the implementation and validation of the previous theoretical results are given. Concluding remarks are stated in Section 5.  Finally, some notation and preliminary results  concerning to stability of time-delay system are stated in the Appendix.


\section{{Description} of the bioreactor}\label{ProblemStatement}

In this section we present the description of the physical plant. It was the desire of modeling and controlling such a plant that motivated the present work. Although these bioprocess are well understood, we were unable to find in the literature a specific mathematical model fitting the data obtained in our experiments.

In Figure~\ref{EsquemaBio} we present a diagram showing a pair of tanks composing the basic structure of our bioreactor. The main tank, where the mixing occurs, is where the inoculum lies. We then supply a percentage $u(t)=D$ of the substrate to enter this tank through a valve from the secondary tank.

\begin{figure}
	\centering
	\includegraphics[scale=.25]{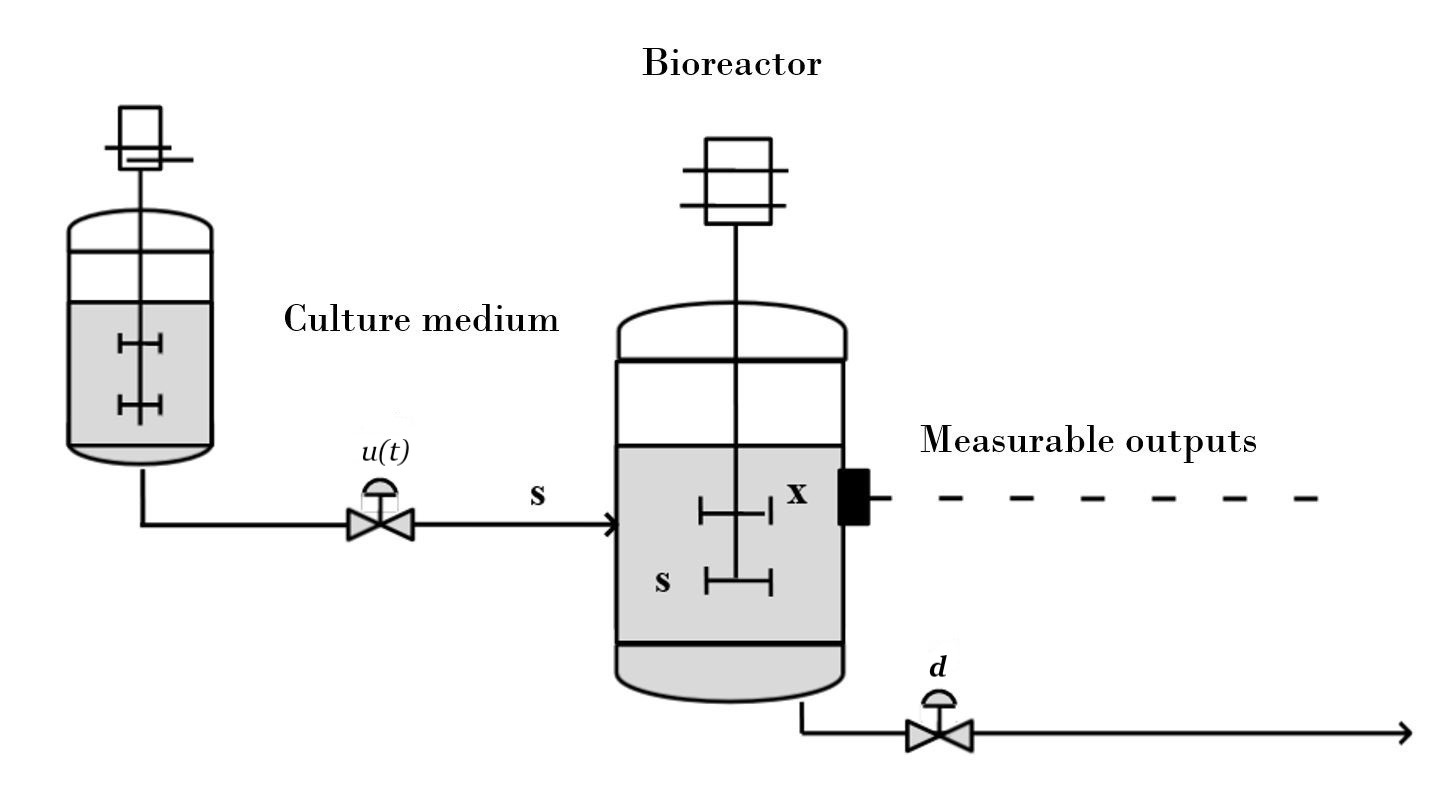}\\
	\caption{Bioreactor diagram}
	\label{EsquemaBio}
\end{figure}

\medskip
\noindent


The fermentation in continuous operation were carried out in a bioreactor containing broth medium 50 ml in a volume of 250 ml, inoculated with colonies of {\it Zymomonas mobilis} for 80 hours.  Samples were incubated at $30\sp\circ$ C and 100 rpm. This experiment was performed in triplicate. Cocoa juice was extracted from cocoa fermentation baskets after two days. Juice samples were immediately stored in sterile bottles. For experiments, cocoa pulp juice was adjusted to 10 g/L of glucose with reducing sugar. Reaction mixture was analyzed by the reducing sugar, applying the 3,5-dinitrosalicyclic acid method \cite{Miller}. The quantity of reducing sugars was calculated using the regression equation, consisting of a standard curve with glucose (1 mg/mL). The bacterial biomass growth was evaluated by dry weight method \cite{APHA}. A 5 mL-aliquot was taken  5 hours  for analyses of biomass and substrate.

Next, a description of the experiment when the maximum biomass an minimal substrate is obtained.
Substrate consumption of cocoa bean pulp during fermentation occurred during the first 25 h of fermentation until a remnant of 0.9 g/L remained. Changes in glucose concentration demonstrated metabolic changes during the fermentation of different hybrids. A decrease in substrate favored the growth of biomass because these microorganisms use this chemical compound as substrate \cite{Moreira}. The maximum biomass concentration in the treatment of glucose to 8.1 g /L was obtained in the first 20 h of a operation  in  continuous fermentation for a dilution rate of 0.15 1/h. The experimental data obtained in the bioprocess is given in Table~\ref{Tab_ExpData}. Note that an oscillatory behaviour is observed in the dynamics of the experimental bioprocess. 

\begin{table}[ht]
		\begin{tabular}{lllll}
		\hline
		\textbf{Time}  &  \textbf{Biomass} & \textbf{Error} & \textbf{Substrate} & \textbf{Error}\\
		$[hours]$ & $[g/L]$ &  $5\%$ & $[g/L]$ & $5\%$ \\
		\hline
		0 & 0.1 & 0.005 & 10.0 & 0.5  \\
		5 & 0.3  & 0.015 & 9.0 & 0.45  \\
		10 & 1.1 & 0.055 & 7.8 & 0.4 \\
		15 &	4.3 & 0.215	& 6.1 & 0.33 \\
		20 & 8.1 & 0.405 & 3.4 & 0.175  \\
		25 & 5.1 & 0.255 & 0.9 & 0.045  \\
		30 & 4.0 & 0.20 & 2.1 & 0.105  \\
		35 & 4.5 & 0.225 & 2.6 & 0.13 \\
		40 & 5.1 & 0.255 & 1.6 & 0.08 \\
		45 & 4.9 & 0.245 & 1.9 & 0.095 \\
		50 & 4.3 & 0.215 & 2.5 & 0.125 \\
		55 & 4.9 & 0.245 & 1.7 & 0.085 \\
		60 & 4.1 & 0.205 & 2.4 & 0.12 \\
		65 & 5.0 & 0.25 & 1.6 & 0.08  \\
		70 & 4.2 & 0.21 & 2.3 & 0.115 \\
		75 & 4.9 & 0.245 & 1.4 & 0.07 \\
		80 & 4.0 & 0.20 & 2.0 & 0.1 \\
		\hline
	\end{tabular}	
	\caption{Experimental data obtained from the biochemical process.}\label{Tab_ExpData}
\end{table}

\section{Description and analysis of mathematical model}
This section presents a description and mathematical analysis of the fractional Lotka-Volterra model with time-delay and delayed controller. First, the model is presented together with a theoretical justification. Next, the analysis of the model with constant excitement is conducted and finally an analysis of the model with the delayed controller, including a description for the tuning of the controller is included.
It is important to emphasize that, in the next section, our model will be shown to fit the experimental data described in the previous section.

\subsection{Description of the fractional Lotka-Volterra model with time delay for a bioreactor}

In this subsection a fractional Lotka-Volterra model with time-dealy for a bioreactor based on the Lotka–Volterra equations is introduced. One reason for choosing this well known predator-prey mathematical model is the appearance of periodic solutions, a phenomena  observed in the experimental data of a bioreactor.  Evidently, the Lotka-Volterra equations alone are not able to capture the rich variety of dynamics involved in a bioreactor. For this purpose, two context specific modifications are proposed: By following an approach gaining prominence in dealing with these complexities, first fractional exponents in the variables have been introduced \cite{Tagh,TannenbaumEtAl}, and secondly, a time-delay term has been added in one of the differential equations
\cite{Forde}. In doing so it turned out the  the mathematical model  adequately  fits the data provided.  

For modeling a bioreactor, the following fractional Lotka-Volterra equations  with time-delay and delayed control is introduced:
\begin{equation}\label{MatModelBio_gral}
\begin{split}
\dot{s}(t) &=  -a s(t)^{\beta} - b s(t)^{\beta}x(t-\tau)^{\alpha}   +  \left(s_0-s(t)\right) u(t),\\
\dot{x}(t) &=   cs(t)^{\beta} {x(t)}^{\alpha}  -d {x(t)}^{\alpha}    + ex(t),
\end{split}
\end{equation}
where $a,\ b,\ c,\ d,\ e\in\erre\sp{+}$ are constant positive parameters of the bioreactor,  $\alpha,\, \beta\in\erre^+$ are positive fractional-exponents, $s_0\in\erre^+$ is the positive initial supply of substrate, $\tau\in\erre^+$ is a non-negative constant delay, $s(t)$ is the amount of subs\-tra\-te  present in the bioreactor at time $t$, $x(t)$ is the amount of such produced biomass at time $t$, and $u(t)\in[0,1]$ is the input signal to manipulate the inflow of substrate into the bioreactor.

When  setting $x(t)=0=u(t)$, the first equation in the model (\ref{MatModelBio_gral}) reduces to $\dot{s}(t)=-a s(t)^{\beta}$. That is to say, in the absence of biomass, the rate of change of the substrate is proportional to a scaling of the substrate by an exponent $\beta$. The negative sign reflects that the substrate in the bioreactor actually decays in the absence of biomass, in contrast to a classical Lotka-Volterra approach where intrinsic rate of prey population increase. It has been observed experimentally  that the value of $\beta$ satisfy the inequality $0<\beta<1$, in particular the substrat, in absence of biomass, is subject to a non-exponential decay. 

On the other hand, for $s(t)=0$, the second equation in the model~(\ref{MatModelBio_gral}) becomes $\dot{x}(t)= -c {x(t)}^{\alpha}    + ex(t)$. That is to say, in the absence of substrate, the biomass follows a decay proposed in the Norton-Simons-Massagu\'e (NSM) model. While the the NSM model actually describes the growth of tumors, the sign of the terms on the right hand side of the NSM model have been reversed  to describe instead the decay of an organism under energy constrains. This approach has been successfully taken in very recent works on the analysis of cancer cells \cite{TannenbaumEtAl} and epidemiological models \cite{Tagh}. Here lies another of the proposed modification of the classical Lotka-Volterra model, where the rate of change of predators in absence of prey is assumed to be proportional to their population. 

The Lotka-Volterra equations are complemented with the terms containing the scaled products which represent the rates of decay and growth provided by the biochemical iterations between substrate and biomass. That is to say,  the terms $bx(t)\sp\alpha s(t)\sp\beta$ and $dx(t)\sp\alpha s(t)\sp\beta$ replace $bx(t)s(t)$ and $dx(t)s(t)$ used in the original model, respectively.  The populations $x(t)$ and $s(t)$ are considered to be scaled by the same exponents $\alpha$ and $\beta$ described above.

Finally, the delay in the variable $x(t)$ in the first equation of the system (\ref{MatModelBio_gral}) is justified as follows: the effect of the biomass in the rate of consumption of  the substrate is affected by a dead time produced either by a gestation  time-delay or by a time delayed measurement. Mathematically, this delay will turn out to produce the possibility for  a stable equilibrium point of the fractional Lotka-Volterra model~(\ref{MatModelBio_gral}), in sharp contrast with the non-stability of equilibrium points of a classical Lotka-Volterra model. Thus the introduction of this time-delay favors the modeling of our bioreactor, where stability of fixed points is observed experimentally. Furthermore, biomass production $x(t)$ can only be manipulated indirectly by supplying substrate $s(t)$, so that the only controllable variable is the substrate $s(t)$ through a control action $u(t)\in[0,1]$ that regulates the flow of the substrate through the bioreactor. Here $u(t)=0$  represents zero flow (valve fully closed) and $u(t)=1$ represents total flow (valve fully open), see Figure~\ref{EsquemaBio}. Therefore, the first equation in (\ref{MatModelBio_gral}) has control action $u(t)$ while the second equation has no control action.

A summary of  the parameters in the fractional Lotka-Volterra model (\ref{MatModelBio_gral}) is shown next.
\begin{itemize}
	\item The parameter $a$ is the intrinsic rate of the amount of scaled substrate [Units: 1/h].
	\item The parameter $\beta$ is the fractional-exponent scaling the amount of substrate [Units: Dimensionless].
	\item The parameter  $b$ is the rate of consumption of the scaled substrate by the scaled biomass [Units: 1/h]. 
	\item The parameter $\alpha$ is the fractional-exponent of growth of the biomass [Units: Dimensionless].
	\item The parameter $c\not=0$  denotes the net rate of growth the scaled biomass in response to the size of the scaled substrate [Units: 1/h].
	\item The parameter $d$ measures the  
	output of the scaled biomass from the bioreactor 
	[Units: 1/h].
	\item The parameter $e$ quantifies anabolism (growth rate) of the biomass. 
\end{itemize}

\subsection{Analysis of mathematical model with constant excitement}

To obtain a parametric identification of the model proposed in (\ref{MatModelBio_gral}), we begin by assuming that the input signal $u(t)=D$ is constant and positive. This amounts to open the valve for the substrate to a constant input flow. Our purpose is to achieve the tuning of the parameters so that our model resembles the experimental data of a bioreactor. In addition, we will give conditions  for obtaining limit cycles and Hopf bifurcations using the time-delay as a bifurcation parameter.

Thus, the mathematical model (\ref{MatModelBio_gral}) is rewritten as
\begin{equation}\label{bio}
\begin{split}
\dot{s}(t) &=  -a s(t)^{\beta} - b s(t)^{\beta} x(t-\tau)^{\alpha}  + D\left(s_0-s(t)\right),\\
\dot{x}(t) &=   cs(t)^{\beta} {x(t)}^{\alpha}  -d {x(t)}^{\alpha}    + ex(t).
\end{split}
\end{equation}
\begin{proposition}\label{Prop_EquiPoint}
	Suppose that the  equation
	\begin{equation}\label{root}
	\left(a + bx^{\alpha} \right)   \left( \frac{d
		- ex^{1-\alpha}}{c}
	\right)  +D \left(  \left( \frac{d- ex^{1-\alpha}}{c} \right) ^{\frac{1}{\beta}}- s_{0} \right)= 0, 
	\end{equation}
has a root $x=x\sp\ast$  such that  $0<x\sp\ast< \left(\frac{c}{e}\right)\sp\frac{1}{1-\alpha}$.
	 Denote 
	\begin{equation*}
	s^{*}=\left(\frac{d- e(x^{*})^{1-\alpha}}{c} \right) ^{\frac{1}{\beta}}>0.
	\end{equation*}
	Then  $z\sp\ast=(s^*,x^*)\sp{\intercal}$ is a positive equilibrium point of the model given in (\ref{bio}).
\end{proposition}

\begin{proof}
	Let us assume that  $x(t)=x(t-\tau)=x^* >0$ and
	$s(t)=s^* >0$ denote the equilibrium point of (\ref{bio}), in particular,  $\dot{s}(t)=0$ and $\dot{x}(t)=0$. Whence, the second equation of (\ref{bio}) becomes
	\begin{equation*}
	0= c {(s^*)}^{\beta} {(x^*)}^{\alpha} -d {(x^*)}^{\alpha} + ex^*.
	\end{equation*}
	or equivalently
	\begin{equation*}
	0= c {(s^*)}^{\beta}  -d  + e{(x^*)}^{1-\alpha}.
	\end{equation*}
	Solving for $s^*$ in the last equation above we obtain
	\begin{equation*}
	s^{*}=\left(\frac{d- e(x^{*})^{1-\alpha}}{c} \right) ^{\frac{1}{\beta}}.
	\end{equation*}
	On the other hand,  by substituting the above expression for $s^*$ in the first equation of
	(\ref{bio}) we get
	\begin{align*}
	0&= -a (s^*)^{\beta} - b(x^*)^{\alpha} (s^*)^{\beta}   - D\left(s^*-s_0\right) \\
	&=    \left(-a - b\left(x^* \right)^{\alpha} \right) (s^*)^{\beta}-D\left(s^*-s_0\right)\\
	&=    \left(-a -b({x}^{*})^{\alpha} \right)  \left(  \left(\frac{d- e(x^{*})^{1-\alpha}}{c} \right) ^{\frac{1}{\beta}} \right)\sp\beta - D \left(  \left(\frac{d- e(x^{*})^{1-\alpha}}{c} \right) ^{\frac{1}{\beta}}- s_{0} \right).
	\end{align*}
	Hence, by solving the following equation for $x=x\sp\ast$ 
	\begin{equation*}
	\left(a +bx^{\alpha} \right)   \left( \frac{d
		- ex^{1-\alpha}}{c}
	\right)  +D \left(  \left( \frac{d- ex^{1-\alpha}}{c} \right) ^{\frac{1}{\beta}}- s_{0} \right)= 0,
	\end{equation*}
	\noindent
	we obtain the desired equilibrium point.
\end{proof}

We remark that in our setting, equation~(\ref{root}) does admit a positive  solution such that $s\sp\ast>0$. For completeness of our analysis, next we provide some algebraic conditions  for such solution to exist, namely, conditions for  the left hand side of equation~(\ref{root}) to change of sign for a couple of values of $x$. Indeed, if $x=\left(\frac{d}{e}\right)^{\frac{1}{1-\alpha}}$ then $d-ex^{1-\alpha}=0$ and so the left hand side of equation~(\ref{root}) reduces to $Ds_0$ which is positive. Similarly if $x=\left(\frac{d-cs_0^\beta}{e}\right)^{\frac{1}{1-\alpha}}$ then $\frac{d-ex^{1-\alpha}}{c}=s_0^\beta$ and so the left hand side of equation~(\ref{root}) reduces to $(-a-bx\sp\alpha)s_0^\beta$. Thus,  a condition for the left hand side of equation (\ref{root}) to be negative and hence for equation~(\ref{root}) to admit a nonzero solution is $-a-bx\sp\alpha<0$. In in addition one requires the inequality $x\sp\ast< \left(\frac{d}{e}\right)\sp\frac{1}{1-\alpha}$ one also guarantees that $s\sp\ast>0$. In what follows, we will assume that the equilibrium points $x\sp\ast$ and $s\sp\ast$ obtained from Proposition~\ref{Prop_EquiPoint} are both positive.

Using Proposition~\ref{Prop_EquiPoint}, the linearization of system (\ref{bio}) is of the form (\ref{GralLineal}), where 
\begin{equation*}
A_0=\begin{pmatrix}
\beta (-a-b x^{\alpha})s^{\beta-1}-D & 0\\
\beta d s^{\beta-1}x^{\alpha} & \alpha (c s^{\beta}-d)x^{\alpha-1}+e
\end{pmatrix}\Big{|}_{
	\begin{array}{l}
	z=z^*\\
	u=u^*
	\end{array}},
\end{equation*}
$A_1=\begin{pmatrix}
0 & -b\alpha s^{\beta} x^{\alpha-1}\\
0 & 0
\end{pmatrix}\Big{|}_{
	\begin{array}{l}
	z=z^*\\
	u=u^*
	\end{array}
}$, $B=(0\ 0)^{\intercal}$ and its quasi-polynomial is
\begin{align}
q(\lambda,\tau)=\text{det}\{\lambda I_2 -A_0-A_1{\rm{e}^{-\lambda\,\tau}} \}={\lambda}^{2}+\kappa_1\,\lambda+\kappa_2+\kappa_3\,{\rm{e}}^{-\lambda\,\tau},\label{qbio}
\end{align}
with
\begin{align*}
\kappa_1&=-\beta (-a-b (x^*)^{\alpha})(s^*)^{\beta-1}+D-\alpha \left(c (s^*)^{\beta}-d\right)(x^*)^{\alpha-1}-e\\
\kappa_2&=\left( \beta (-a-b (x^*)^{\alpha})(s^*)^{\beta-1}-D\right)\left(\alpha \left(c (s^*)^{\beta}-d\right)(x^*)^{\alpha-1}+e \right)\\ \kappa_3&=c\,b\,\alpha\,\beta\,(s^*)^{2\beta-1}\,(x^*)^{2\alpha-1}.
\end{align*}
\medskip
\noindent
The stability of system (\ref{bio}) is completely determined by the location of the roots of its corresponding characteristic quasi-polynomial (\ref{qbio}). D-partition me\-thod proposed by Neimark
in \cite{neimark1973d} determine stability conditions of a quasi-polynomial through study of the space of crossover frequencies $i\omega$-crossing delays. Below, we employ this method to the quasi-polynomial (\ref{qbio}). In addition, we will assume the the system is initially stable, that is, stable when $\tau=0$.
\medskip

\noindent
A stable quasi-polynomial loses  stability if some of its roots cross to the open right-half of  the complex plane. Clearly, the above occurs when the roots first cross the imaginary axis. For this, there are two possible cases: i) purely imaginary crossover window $\lambda=\pm i\omega$, where $0\neq\omega\in\erre^+$, ii)  crossover window on the origin $\lambda=0$. In both cases, $\lambda$ must be solution of quasi-polynomial. In general, the crossover window occur under variations of the parameters of a system or quasi-polynomial. A particular case and of great interest to the scientific community since it is closely related to bifurcation theory, it is to find the crossover windows when delay $\tau$ varies.  Next, an analysis of the quasi-polynomial (\ref{qbio}) using the mentioned above is performed.

\medskip
\noindent
Consider the change of variable $\lambda=0$ in the quasi-polynomial (\ref{qbio})
\begin{align*}
q(0,\tau)=&\kappa_2+\kappa_3=0.
\end{align*}
Clearly, we cannot obtain stability conditions from the previous equation
whence, efforts will be focused when $\lambda= i\omega$, $0\neq\omega\in\erre^+$, is solution of quasi-polynomial (\ref{qbio}), 
\begin{align*}
q(i\omega,\tau)
=&{(i\omega)}^{2}+\kappa_1\,i\omega+\kappa_2+\kappa_3\,{\rm{e}}^{-i\omega\,\tau}\notag\\
= & -{\omega}^{2}+\kappa_2 +\kappa_3\cos\left(\omega\tau\right)
+\left(\kappa_1\,\omega- \kappa_3\,\sin \left( \omega \tau  \right) \right)\,i=0.    
\end{align*}
Note that, $q(i\omega,\tau)=0$ iff  
\begin{align*}
\mathrm{Re}\{q(i\omega,\tau) \} &=  -{\omega}^{2}+\kappa_2 +\kappa_3\cos\left(\omega\tau\right)=0,\\
\mathrm{Im}\{q(i\omega,\tau) \} &= \kappa_1\,\omega- \kappa_3\,\sin \left( \omega \tau  \right)=0.
\end{align*}
For the above, we have 
$$\cos \left( \omega \tau  \right) = {\frac {{\omega}^{2}-\kappa_2}{\kappa_3 }},\quad \sin \left( \omega \tau  \right) = {\frac {\kappa_1\,\omega}{\kappa_3}}. $$
Therefore, the following is satisfied
\begin{align}
0 &=\sin^{2}(\omega \tau ) +\cos^{2}(\omega \tau ) - 1= \left(\frac {\kappa_1\,\omega}{\kappa_3}\right)^{2}+\left(\frac {{\omega}^{2}-\kappa_2}{\kappa_3 }\right)^2-1\notag\\
&= {\omega}^{4}+( \kappa_1^2-2\kappa_2){\omega}^{2}
+\kappa_2^2-\kappa_3^2=P(\omega).    \label{polbioo}
\end{align}
Thus, the quasi-polynomial (\ref{qbio}) has roots $ \lambda_{0} = \pm i \omega_{0}$, if 
$ \omega_{0}$ is solution of the polynomial $P(\omega_{0})$ given in (\ref{polbioo}). Moreover, 
the delay value where the above occurs is
\begin{equation}\label{tau0}
{ \tau  }_{ 0 }=\frac { 1 }{ { \omega_0  } } {\tan }^{ -1 }\left(\frac {\kappa_1\omega_0 }{ \omega_0^2-\kappa_2} \right) + \dfrac{n\pi}{\omega_0};\ n=0,\pm 1, \pm2, \ldots ,
\end{equation}
Next, we postulate in the following criteria the stipulated above.
\begin{proposition}\label{Prop_Cross}
	The quasi-polynomial (\ref{qbio}) has crossover windows $\lambda_{0} = \pm i \omega_{0}$, if there exists $\omega_0>0$ such that $P(\omega_{0})=0$ given in (\ref{polbioo}).
	In addition, the values of the delay $\tau$ where the above occurs satisfy the equation given in
	(\ref{tau0}).
\end{proposition}
The previous result provides conditions of the delay $\tau_0$ for which the quasi-polynomial (\ref{qbio}) has roots on the imaginary axis $\lambda_{0} = \pm i \omega_{0}$. If these roots are dominant, then it is natural to determine when the roots move to the right half-plane under infinitesimal variations of this delay $\tau_0$. The above is know as direction of the crossing \cite{walton1987direct} and it is determined using the following result.
\begin{proposition}\label{Prop_sign}
	The mathematical model given in (\ref{bio}) has a Hopf bifurcation  at $\tau=\tau_0$ if 
	\begin{equation}\label{sign}
	\mathrm{sign}\Bigl\{\mathrm{Re}\Bigl\{ \frac{\partial \lambda}{\partial\tau}\Big{|}_{\lambda=i\omega_0}\Bigl\} \Bigl\}=\mathrm{sign}\left\{ {\kappa_1}^{2}-2\,\kappa_2 \right\}>0.
	\end{equation}
	Here, $\tau_0$ is given in (\ref{tau0}).
\end{proposition}
\begin{proof}
	The technique will need to check for each root on 
	the imaginary axis $\lambda_{0} = \pm i \omega_{0}$
	crosses to the left or to the right of the complex plane
	when increasing $\tau$. This is determined by the sign of
	$\text{Re}\{\partial \lambda/\partial\tau\}$. Consider the quasi-polynomial given in (\ref{qbio}) 
	and suppose that $\lambda=\lambda(\tau)$, then
	$$\frac{\partial \lambda}{\partial \tau} = \frac{ \lambda\kappa_3 }{ \left( 2\lambda+\kappa_1 \right){\rm{e}}^{\lambda\,\tau} -\tau\kappa_3  },$$
	which implies that
	\begin{align*}
	\left( \frac{\partial \lambda}{\partial \tau} \right)^{-1}&=\left(\frac{2\lambda+\kappa_1}{\lambda\kappa_3} \right){\rm{e}}^{\lambda\,\tau}-  \frac{\tau\kappa_3}{\lambda\kappa_3}\\
	&=\left(\frac{2\lambda+\kappa_1}{\lambda\kappa_3} \right)\left(
	-\frac{\kappa_3}{{\lambda}^{2}+\kappa_1\,\lambda+\kappa_2}
	\right)- \frac{\tau}{\lambda}\\
	&=-{\frac {\kappa_1+2\,\lambda}{\lambda\, \left( \kappa_1\,\lambda+{
				\lambda}^{2}+\kappa_2 \right) }}-{\frac {\tau}{\lambda}}.
	\end{align*}
	Therefore, for $\lambda=i \omega_0$ yields
	\begin{align*}
	\left( \frac{\partial \lambda}{\partial \tau} \right)^{-1}\Big{|}_{\lambda=i\omega_0}
	=&{\frac {2\,{\omega_{0}}^{2}+{\kappa_1}^{2}-2\,\kappa_2}{{\omega_{0}}^{4}+
			\left( {\kappa_1}^{2}-2\,\kappa_2 \right) {\omega_{0}}^{2}+{\kappa_2}^{2}}}\\
	&+{\frac { \left( \tau\,{\omega_{0}}^{4}+ \left( {\kappa_1}^{2}\tau-2\,\tau
			\,\kappa_2+\kappa_1 \right) {\omega_{0}}^{2}+\kappa_2\, \left( \tau\,\kappa_2+
			\kappa_1 \right)  \right) }{\omega_{0}\, \left( {\omega_{0}}^{4}+ \left( {
				\kappa_1}^{2}+2\,\kappa_2 \right) {\omega_{0}}^{2}-{\kappa_2}^{2} \right) }}i.
	\end{align*}
	Thus, for $\omega_{0}>0$ we get 
	\begin{align*}
	\text{sign}\Bigl\{\text{Re}\Bigl\{ \frac{\partial \lambda}{\partial\tau}\Big{|}_{\lambda=i\omega_{0}}\Bigl\} \Bigl\}&=\text{sign}\Bigl\{\text{Re}\Bigl\{ \left(\frac{\partial \lambda}{\partial\tau}\right)^{-1}\Big{|}_{\lambda=i\omega_0}\Bigl\} \Bigl\}\\
	&= \text{sign}\left\{ {\kappa_1}^{2}-2\,\kappa_2 \right\}. 
	\end{align*}
	This ends the proof.
\end{proof}
Note that the sign given in (\ref{sign}) is independent of $\tau_0$, therefore if the linearization of system (\ref{bio}) is stable for $\tau=0$, then it remains stable for all $\tau\in(0,\tau_0)$, where $\tau_0$ is the first value of (\ref{tau0}). Thus the following result is evident.
\begin{proposition}\label{Prop_stability}
	The quasi-polynomial (\ref{qbio}) given in (\ref{bio}) is stable for all $\tau\in(0,\tau_0)$ and unstable for all $\tau>\tau_0$, where $\tau_0$ is the first value such that satisfies the equation given in (\ref{tau0}).
\end{proposition}
\begin{proof}
	The result follows using Propositions~\ref{Prop_Cross} and \ref{Prop_sign}. 
\end{proof}

\subsection{Analysis of mathematical model with delayed controller}

In this section, we carry out a stability analysis of the fractional mathematical model (\ref{MatModelBio_gral}) in closed-loop with a delayed controller of the type (\ref{GralControl}). As well as tuning the controller gains.
\begin{proposition}\label{PropEqui2}
	Let $x\sp\ast$ be a positive real number such that
	\begin{equation*}
	0<s\sp\ast =\left(\frac{d-e\,(x\sp\ast)^{1-\alpha}}{c} \right)^{\frac{1}{\beta}} < s_0 \hspace*{0.5cm}
	\text{ and }
	\hspace*{0.5cm}
	a+b(x^\ast)\sp\alpha \leq\frac{s_0-s\sp\ast}{(s\sp\ast)^\beta}.
	\end{equation*} 
    If
	\begin{equation*}
	u\sp\ast= \frac{\left(a+b(x\sp\ast)\sp\alpha\right)(s\sp\ast)\sp\beta}{s_0-s\sp\ast},
	\end{equation*}
	then $z\sp\ast=(s^*,x^*)\sp{\intercal}$ and $u\sp\ast$ are the equilibrium point of the fractional mathematical model (\ref{MatModelBio_gral}).
\end{proposition}

\begin{proof}
	Let us assume that $0<s(t)=s\sp\ast<s_0$, $0<x(t)=x\sp\ast$ and $0<u(t)=u\sp\ast$ denote the equilibrium point of the  model (\ref{MatModelBio_gral}). Then the second equation of (\ref{MatModelBio_gral}) becomes
	\begin{equation*}
	c(s\sp\ast)^\beta(x\sp\ast)^\alpha -d(x\sp\ast)^\alpha + e x\sp\ast =0,
	\end{equation*}
	\noindent
	and solving for $s\sp\ast$ we obtain
	\begin{align*}
	s\sp\ast =\left(\frac{d-e\,(x\sp\ast)^{1-\alpha}}{c} \right)^{\frac{1}{\beta}}.
	\end{align*}
	Now, the first  equation of (\ref{MatModelBio_gral}) we obtain
	\begin{align*}
	u\sp\ast= \frac{\left(-a-b(x\sp\ast)\sp\alpha\right)(s\sp\ast)\sp\beta}{s\sp\ast-s_0}.
	\end{align*}
	The hypothesis imply that $0\leq u\sp\ast\leq 1$.
	Thus for the fixed positive value $x\sp\ast$, and the value of $s=s\sp\ast$ satisfying the given inequalities, the value of  $u=u\sp\ast$ is obtained so that they are equilibrium point, as was to be shown.
\end{proof}

\medskip
Using Proposition~\ref{PropEqui2}, the linearization of model (\ref{MatModelBio_gral}) is of the form (\ref{GralLineal}), where
\begin{equation} \label{Matri_LinealControl}
\begin{split}
A_0=
\begin{pmatrix}
\beta(-a-bx\sp\alpha)s^{\beta-1}-u
& 0\\
\beta cs^{\beta-1}x\sp\alpha &
\alpha\left( cs\sp\beta-d \right)x^{\alpha-1}+e
\end{pmatrix}\Big{|}_{
	\begin{array}{l}
	z=z^*\\
	u=u^*
	\end{array}},\\
A_1=\begin{pmatrix}
0 & -\alpha\ b\ s^{\beta}x^{\alpha-1}\\
0 & 0
\end{pmatrix}\Big{|}_{
	\begin{array}{l}
	z=z^*\\
	u=u^*
	\end{array}},\ \text{and}\
B=\begin{pmatrix}
s_0-s\\ 0
\end{pmatrix}\Big{|}_{
	\begin{array}{l}
	z=z^*\\
	u=u^*
	\end{array}}.
\end{split}
\end{equation}
Next, we propose the following delayed controller to stabilize the model
\begin{equation}\label{ControlPR}
u(t)=Kz(t-h)=(0,\ k_{r})\begin{pmatrix}
s(t-h)\\
x(t-h)
\end{pmatrix},
\end{equation}
where $k_{r}\in\erre$ is the controller gain and $h\in\erre^{+}$ is the controller delay. 
Note that in  (\ref{ControlPR}) the  proposed  controller is only acting on the biomass $x$ and not on the substrate $s$. We base our proposal in the fact that,  during the experiments, we only have access to the measurement of $x$ and not of $s$.
In addition, the time-delay in the controller is justified by the  time-delay in between the recordings of  the biomass output.

\medskip
Let $A_0,\ A_1\in\erre^{2\times2}$ and $B\in\erre^2$ given in (\ref{Matri_LinealControl}), then the quasi-polynomial of closed-loop systems (\ref{GralLineal}) with (\ref{ControlPR}) is 
\begin{align}\label{quasi_control}
q(\lambda,\tau,h)&=\text{det}\{\lambda I_2-A_0-A_1{\rm e}^{-\tau\lambda}-BK{\rm e}^{-h\lambda} \}\notag\\
&=\lambda^2+\eta_1\lambda+\eta_2+\eta_3{\rm e}^{-\tau\lambda}+\eta_4\,k_{r}\, {\rm e}^{-h\lambda},
\end{align}
where 
\begin{equation}\label{etas}
\begin{split}
\eta_1&=-\beta(-a-b(x^*)\sp\alpha)(s^*)^{\beta-1}+u^*-\alpha\left( c(s^*)\sp\beta-d \right)(x^*)^{\alpha-1}-e,\\
\eta_2&=\left( \beta(-a-b(x^*)\sp\alpha)(s^*)^{\beta-1}-u^*\right)\left( \alpha\left( c(s^*)\sp\beta-d \right)(x^*)^{\alpha-1}+e\right),\\ 
\eta_3&=c\ b\ \alpha\ \beta\  (s^*)^{2\beta-1}(x^*)^{2\alpha-1},\\
\eta_4&=\left( \beta c(s^*)^{\beta-1}(x^*)\sp\alpha \right)\left(s_0-s^* \right).    
\end{split}    
\end{equation}

\subsubsection{Tuning of the delayed controller }

Here again we use the D-partition method proposed by Neimark \cite{neimark1973d} and the tuning method proposed in \cite{villafuerte2018tuning}. Consider the quasi-polynomial given in (\ref{quasi_control}) and the chance of variable $\lambda=\lambda-\sigma$, for a given $\sigma>0$, i.e.
\begin{align}\label{quasi_control_sigma}
q_{\sigma}(\lambda,\tau,h)&=q(\lambda-\sigma,\tau,h)\\
&=(\lambda-\sigma)^2+\eta_1(\lambda-\sigma)+\eta_2+\eta_3{\rm e}^{-\tau(\lambda-\sigma)}+\eta_4\,k_{r}\, {\rm e}^{-h(\lambda-\sigma)}=0.\notag
\end{align}
If $\lambda=0$ note that the above quasi-polynomial is 
\begin{align*}
q_{\sigma}(0,\tau,h)&=\sigma^2-\eta_1\sigma+\eta_2+\eta_3{\rm e}^{\tau\sigma}+\eta_4\,k_{r}\, {\rm e}^{h\sigma}=0,
\end{align*}
therefore
\begin{equation*}
k_{r}=\frac{\sigma^2-\eta_1\sigma+\eta_2+\eta_3{\rm e}^{\tau\sigma}}{\eta_4\,{\rm e}^{h\sigma}}.
\end{equation*}
If $\lambda=i\omega$ note that the quasi-polynomial (\ref{quasi_control_sigma}) is 
\begin{align*}
q_{\sigma}(i\omega,\tau,h)&=(i\omega-\sigma)^2+\eta_1(i\omega-\sigma)+\eta_2+\eta_3{\rm e}^{-\tau(i\omega-\sigma)}+\eta_4\,k_{r}\, {\rm e}^{-h(i\omega-\sigma)}
=0,
\end{align*}
the previous equality is satisfied if
\begin{align}
\text{Re}\{ q_{\sigma}(i\omega,\tau,h)\}&=\Phi+k_{r}\,\eta_4\cos(h\omega){\rm e}^{h\sigma}=0,\label{Eq_Re}\\
\text{Im}\{ q_{\sigma}(i\omega,\tau,h)\}&=\Theta-k_{r}\,\eta_4\sin(h\omega) {\rm e}^{h\sigma}=0,\label{Eq_Im}
\end{align}
where 
\begin{equation}
\begin{split}\label{PhiTheta}
\Phi&=\sigma^2-\omega^2-\eta_1\sigma+\eta_2+\eta_3\cos(\tau\omega){\rm e}^{\tau\sigma}\ \text{and}\\ \Theta&= \eta_1\omega-2\omega\sigma-\eta_3\sin(\tau\omega){\rm e}^{\tau\sigma}.    
\end{split}    
\end{equation}
Solving for $k_r$ of (\ref{Eq_Im}) we obtain
\begin{align*}
k_{r}=\frac{\Theta}{\eta_4\sin(h\omega) {\rm e}^{h\sigma}}.
\end{align*}
By substituting the above equation in (\ref{Eq_Re}) we get
\begin{align*}
0&=\Phi+\left( \frac{\Theta}{\eta_4\sin(h\omega) {\rm e}^{h\sigma}}\right)\eta_4\cos(h\omega){\rm e}^{h\sigma}.
%
\end{align*}
Solving $h$ of above equation
\begin{align*}
h=\frac{1}{\omega}\cot^{-1}\left( -\frac{\Phi}{\Theta}  \right).
\end{align*}
Next, we postulate a result for tuning the delayed controller (\ref{ControlPR}) using the above.
\begin{proposition}\label{Prop_Regions}
	Consider the quasi-polynomial (\ref{quasi_control})
	and let $\sigma>0$ be given. Then the $\sigma$-stability regions of the quasi-polynomial (7) on the parametric plane $h-k_r$ are delimited by the following equations
	\begin{equation*}
	k_{r}:=k_r(h)=\frac{\sigma^2-\eta_1\sigma+\eta_2+\eta_3{\rm e}^{\tau\sigma}}{\eta_4\,{\rm e}^{h\sigma}},\quad\text{when}\ \lambda=0;
	\end{equation*}
	and when $\lambda=i\omega$, $\omega>0$,
	\begin{equation*}
	\begin{split}
	\hat{h}&:=h(\omega,\sigma)=\frac{1}{\omega}\cot^{-1}\left( -\frac{\Phi}{\Theta}  \right)+\frac{n\pi}{\omega},\ n=0,\pm 1,\pm 2,\dots;\\
	\hat{k}_{r}&:=k_r(\hat{h},\omega,\sigma)=\frac{\Theta}{\eta_4\sin(\hat{h}\omega) {\rm e}^{\hat{h}\sigma}},
	\end{split}
	\end{equation*}
	where $\eta_j$, $j=1,\dots,4$ are given in (\ref{etas}) and $\Theta$, $\Phi$ in (\ref{PhiTheta}).
\end{proposition}

\section{Validation of theoretical results}

\subsection{Validation of mathematical model}

As mentioned before, for parametric identification, in this section we assume that $u(t)=D$ in the mathematical model (\ref{MatModelBio_gral}). In addition, experimental data is taken from a biochemical process as reference points to identification, see Table~\ref{Tab_ExpData}. For the simulations, we consider initial conditions of $s_0=10$ g/L  for the substrate and $x_0=0.1$  g/L for the biomass, as well as a constant opening of the substrate flow valve equal $D=0.15$ 1/h.  These initial parameters are the same as those implemented in the biochemical process described above.  To obtain the values of the other parameters  nonlinear regression was used based on the Levenberg-Marquardt least squares minimization algorithm. In consequence, the values obtained from our model come closer to the experimental data, see Figure~\ref{Identication} and \ref{IdenticationFhase}. They are
\begin{equation}\label{Parameters}
\begin{array}{ccccc}
a =0.16, & b = 0.11, & c=0.282, & d=0.47, & e=0.212, \\
\alpha=1.3, & \beta = 0.27, & s_0=10,  & \text{and} & \tau=1.8. 
\end{array}
\end{equation}
Whereby, $\kappa_1=0.37985$, $\kappa_2=0.02011$ and $\kappa_3=0.09791$.  The parameter values obtained in the present study fall within the range of those reported in the literature, due to the different operating conditions used in each case, i.e., different carbon source, continuous or batch operation, temperature, pH, among others (\cite{Buehler, Paz, Rama, Zhang}).

\begin{figure}[ht]
	\centering
	\includegraphics[scale=.45]{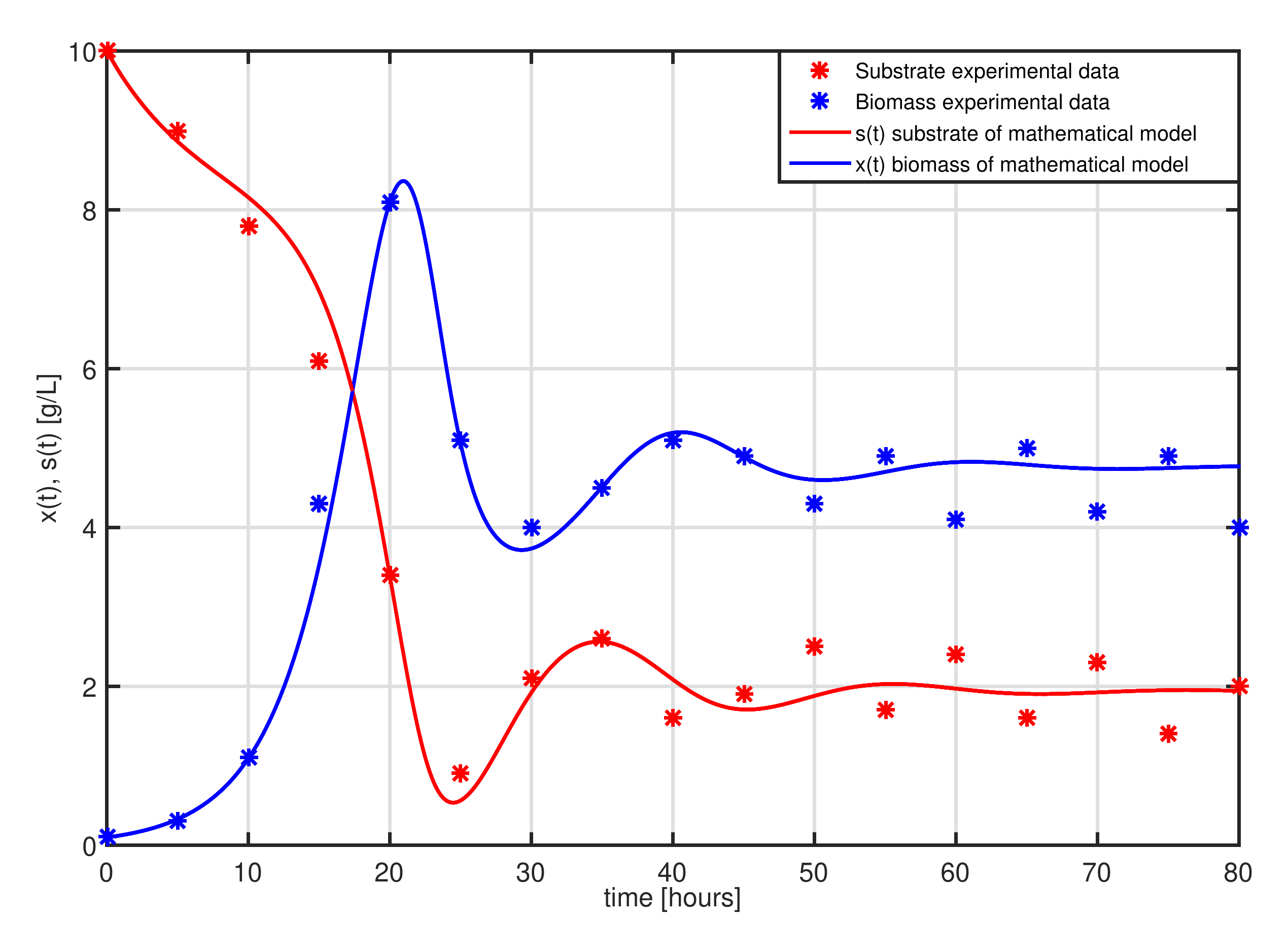}\\
	\caption{Parametric identification of (\ref{bio}) using experimental data when $D=0.15\, 1/h$ and $s_0=10\, g/L$ are given.}\label{Identication}
\end{figure}
\begin{figure}[ht]
	\centering
	\includegraphics[scale=.45]{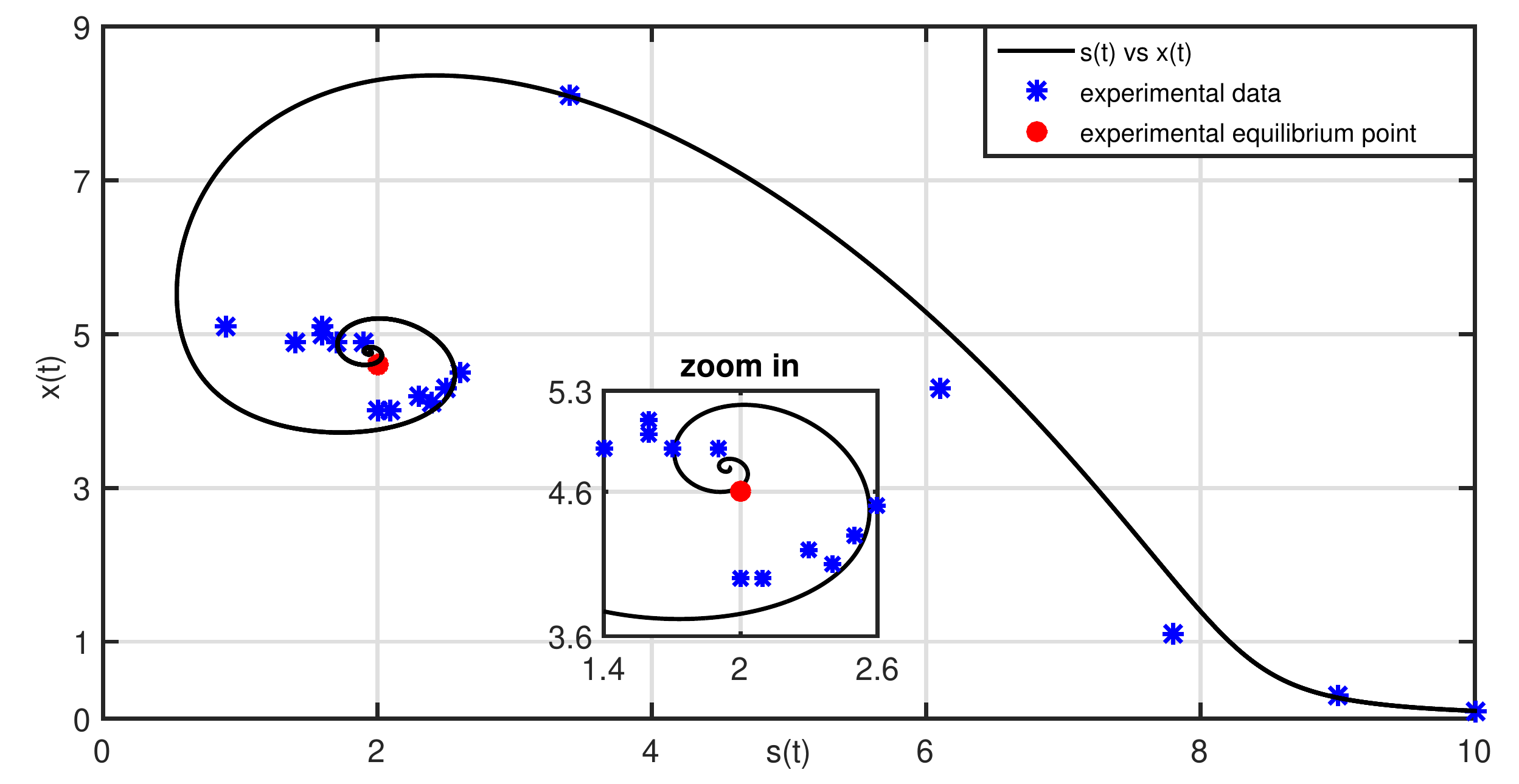}\\
	\caption{Phase diagram of mathematical model (\ref{bio}) and experimental data given in Table~\ref{Tab_ExpData}.}\label{IdenticationFhase}
\end{figure}

The performance of the proposed mathematical model was statistically validated using the dimensionless coefficient of efficiency $(\epsilon_1)$, where
\begin{equation*}
\epsilon_1=1-\frac{\sum_{i=1}^{N}
	\left|Y-Y\sp\ast\right|}
{\sum_{i=1}^{N}
	\left|Y\sp\ast-\overline{Y}\right|},
\end{equation*}
\noindent
$-\infty<\epsilon_1<1$, $Y$ is the simulated value of the variable at time $t_i$, $Y\sp\ast$ is the observed value of the same variable at time $t_i$ and $\overline{Y}$ the mean value of the observed variable. A positive value of $\epsilon_1$ represents an acceptable simulation, whereas $\epsilon_1>0.95$ represents good simulation \cite{Cuevas,Hecke}.

It is recommended, therefore, to use ($\epsilon_1$) in lieu of correlation-based measures to provide a relative assessment of model performance. The statistics used absolute values rather than squared differences. Parameter efficiency ($\epsilon_1$) for all variables using the L-M method are: Biomass (0.850) and Substrate (0.834) and, Interpretation of correlation-based measures, 0.85 indicate that the model explains 85.0\% of the variability in the observed data. Therefore, the proposed model was able to predict experimental data.

Now, using Proposition~\ref{Prop_EquiPoint} and parameters given in (\ref{Parameters}) an equilibrium point of the mathematical model (\ref{bio}) is 
\begin{equation}\label{Equi_SinControl}
x^*= 4.77631\quad \text{and}\quad s^*=1.9427.
\end{equation}
Therefore, the linearization of (\ref{bio}) using (\ref{Parameters})  and (\ref{Equi_SinControl}) is 
\begin{equation*}
A_0= \begin{pmatrix}
-0.3162 & 0\\
0.3580 & -0.0636
\end{pmatrix},\quad\text{and}\quad %
A_1= \begin{pmatrix}
0 & -0.2734\\
0 & 0
\end{pmatrix};
\end{equation*}
and by (\ref{qbio}) its quasi-polynomial is
\begin{equation*}
q(\lambda,\tau)={\lambda}^{2}+ 0.37985\,\lambda+ 0.02011+ 0.09791\,{
	{\rm e}^{-\tau\,\lambda}}.
\end{equation*}
Thus, the polynomial (\ref{polbioo}) gives
\begin{equation*}
P(\omega)={\omega}^{4}+ 0.10406\,{\omega}^{2}- 0.0091818.
\end{equation*}
Finally, using Proposition~\ref{polbioo} and the solution $\omega_0=0.23876$ of the above polynomial, we get some candidate values of delay $\tau=\tau_0$ which may be critical values of $\tau_0$, where the mathematical model (\ref{bio}) has a Hopf bifurcation. Some calculated critical values are
\begin{equation}\label{tau0_exp}
\tau_0=4.9608,\ 18.1187,\ 31.2767\ \text{and}\ 44.4346.  
\end{equation}
These values were obtained using the equation (\ref{tau0}) for $n=0,1,2,3,4$. To corroborate that the points given in (\ref{tau0_exp}) are critical points, in Figure~\ref{Phase_sinControl} the phase diagrams of model (\ref{bio}) are presented. Note that for initial conditions $(s_0=1.85,\, x_0=4.9)$ near the equilibrium point (\ref{Equi_SinControl}), even more the system response (\ref{bio}) remains oscillating around the equilibrium point, thus verifying the existence of a limit cycle on $\tau_0=4.9608$.

\begin{figure}[ht]
	\centering
	\includegraphics[scale=.45]{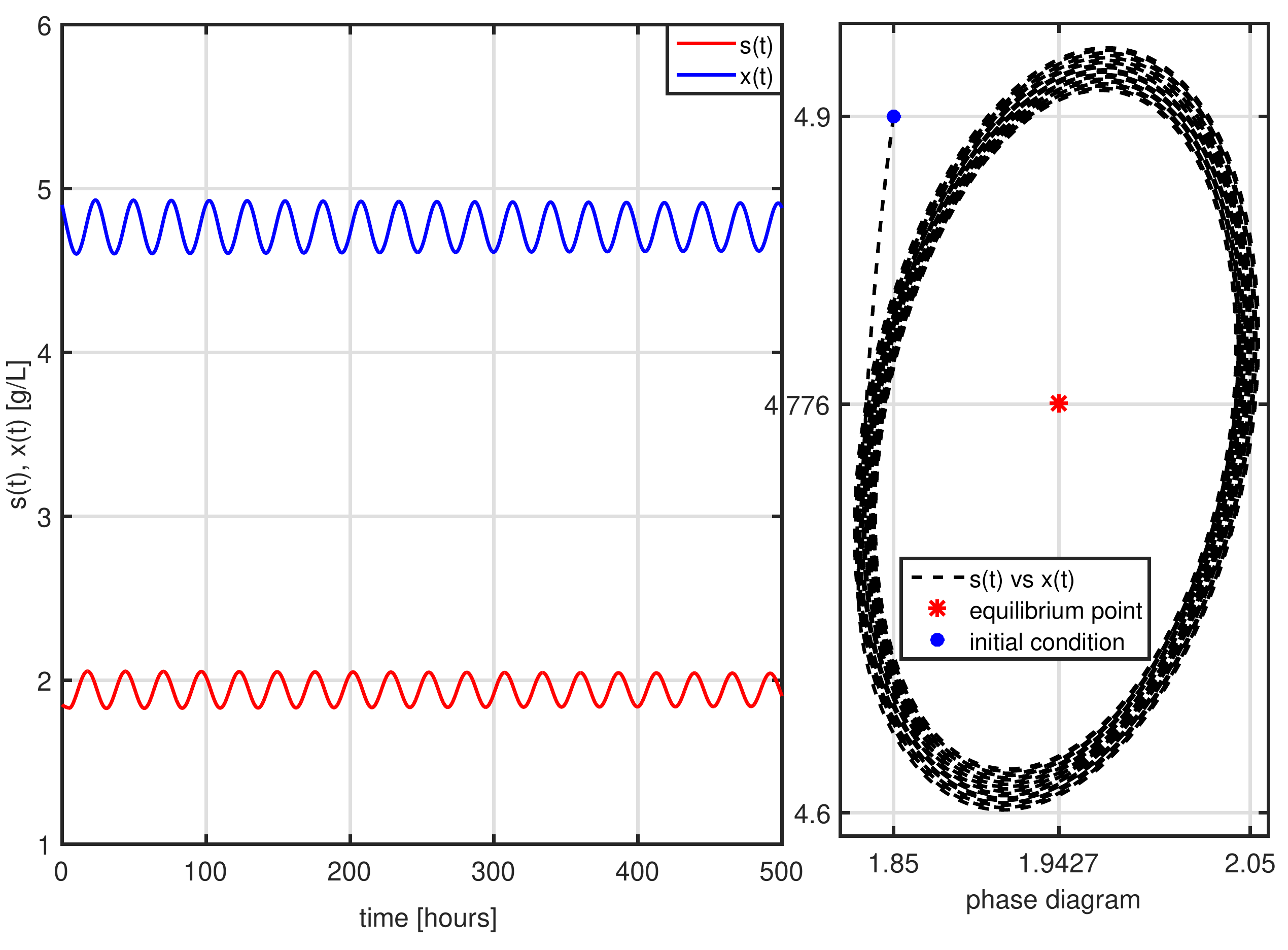}\\
	\caption{System response and phase diagrams of model (\ref{bio}) for  $\tau_0=4.9608$ given in (\ref{tau0_exp}). Here $s_0=1.85$ and $x_0=4.9$.}\label{Phase_sinControl}
\end{figure}   

While in Figure~\ref{Phase_Tau_varios} the stability of the model (\ref{bio}) for $\tau=1,\dots,4$ is shown, using phase diagrams. Note that for the same initial condition $(s_0=1.85,\, x_0=4.9)$, but different values of $\tau$, the system response converges to the equilibrium point (\ref{Equi_SinControl}), corroborating the stability of (\ref{bio}) for all $\tau\in(0,\tau_0)$, as postulated in the Proposition~\ref{Prop_stability}.
\begin{figure}[ht]
	\centering
	\includegraphics[scale=.65]{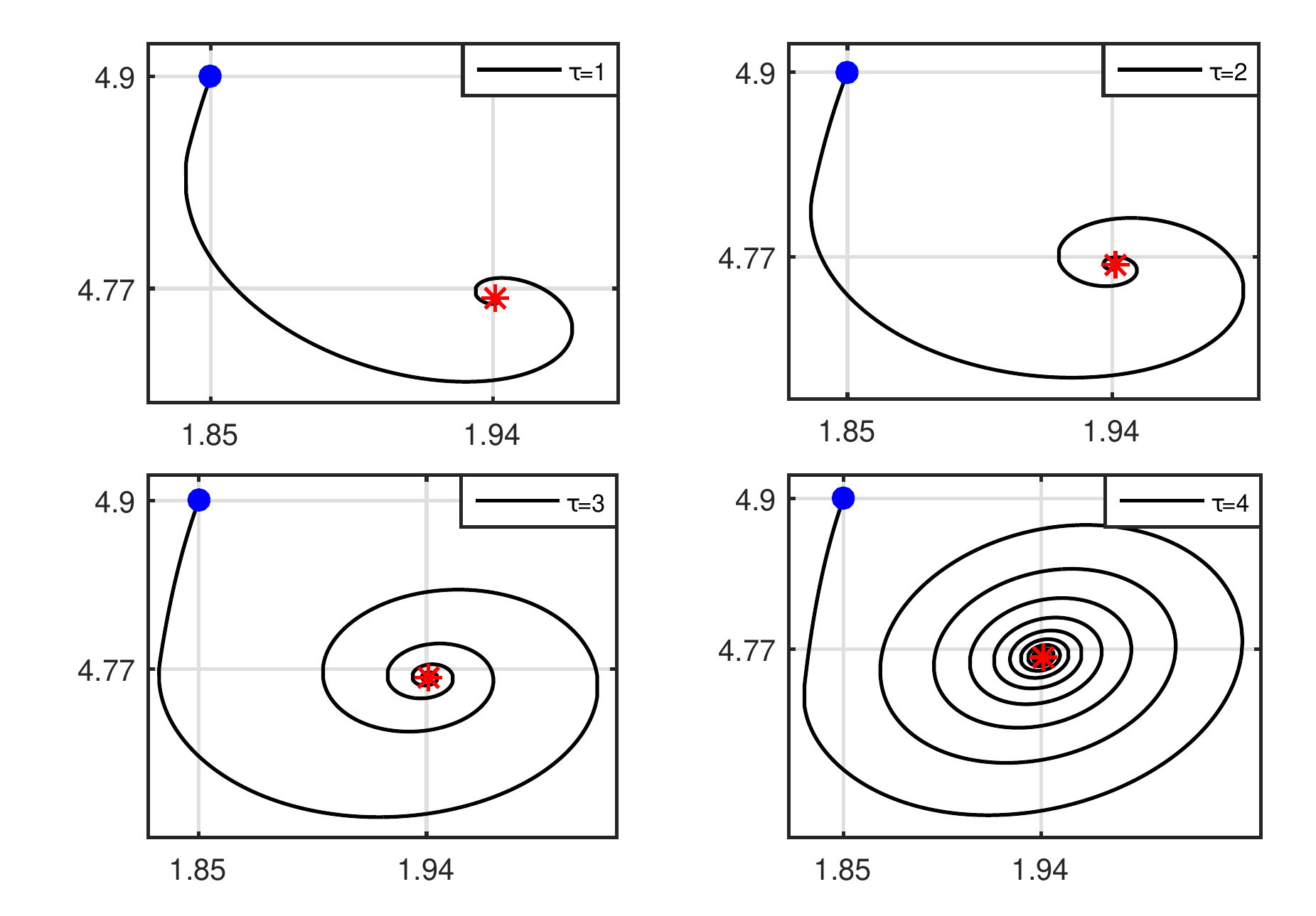}\\
	\caption{Phase diagrams $s(t)$ vs $x(t)$ of model (\ref{bio}) when $\tau=1,2,\dots,4$. Here $(s_0,\,x_0)=(1.85,\, 4.9)$ and $(s^{*},\,x^{*})=(1.94,\,4.77)$.}\label{Phase_Tau_varios}
\end{figure}  
On the other hand, by Proposition~\ref{Prop_sign} we have that (\ref{sign}) is $$\text{sign}\{\kappa_1^{2}-2\kappa_2\}=\text{sign}\{0.37985^2-2(0.02011)\}=\text{sign}\{0.10406\}>0.$$ Thus, in Figure~\ref{Phase_sinControl_tau=7}, the system response and a phase diagram of (\ref{bio}) are presented,  when $\tau=7$. Clearly, the solution $(s(t),\, x(t))$ diverges from the equilibrium point (\ref{Equi_SinControl}), showing via simulation that the model (\ref{bio}) is unstable, when $\tau>\tau_0$ and the initial condition is nearby the equilibrium point and corroborating the provisions of the Proposition~\ref{Prop_stability}.
\begin{figure}[ht]
	\centering
	\includegraphics[scale=.45]{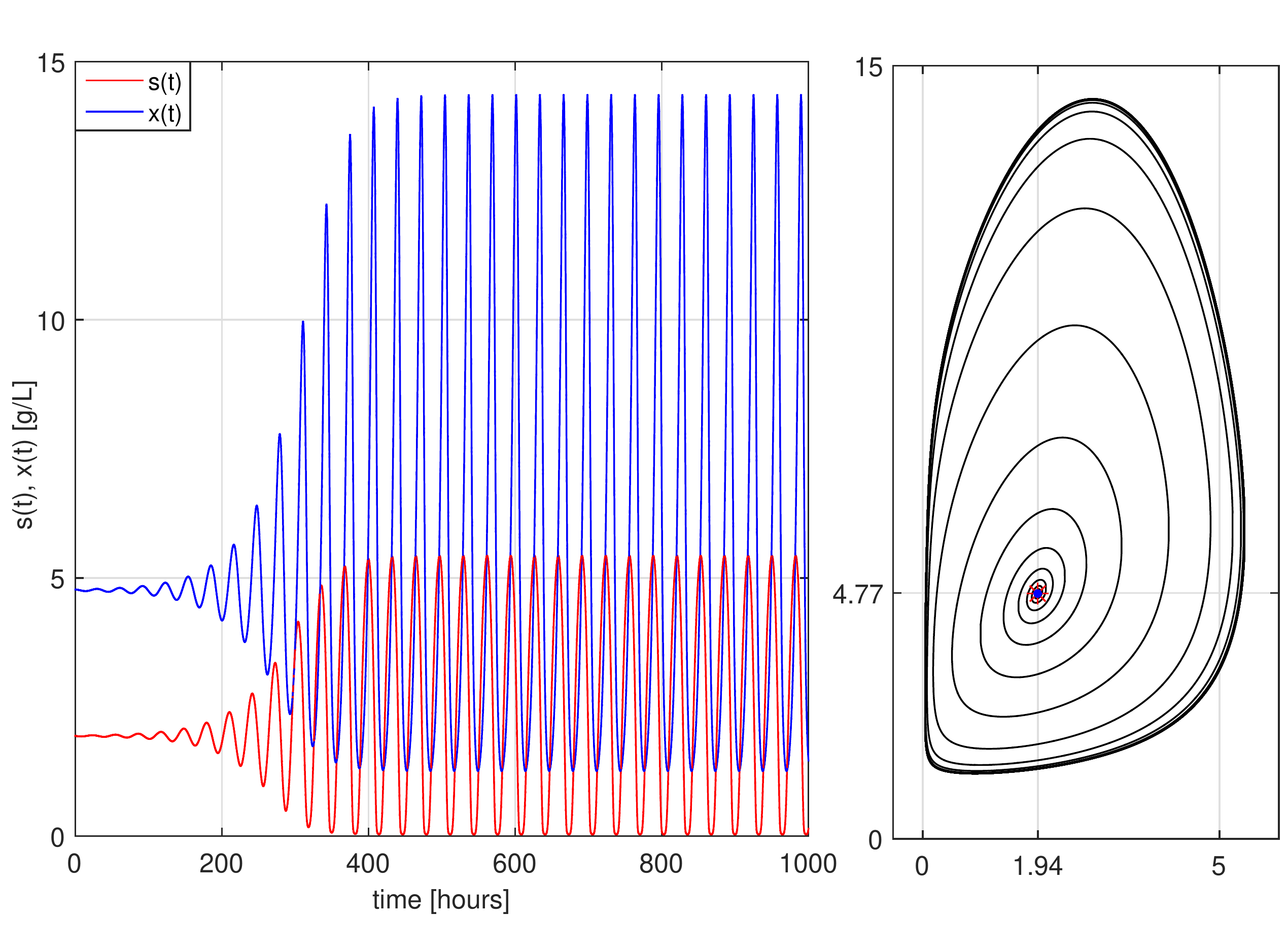}\\
	\caption{System response and phase diagrams of model (\ref{bio}) when  $\tau=7$ and $(s_0,\,x_0)=(1.94,\,4.77)$.}\label{Phase_sinControl_tau=7}
\end{figure}

\subsection{Implementation of delayed controller}

Now, we stabilize the mathematical model (\ref{MatModelBio_gral}) using delayed controllers  of the form (\ref{ControlPR}). 

\medskip
Consider the parameters given in (\ref{Parameters}), $\tau=7$. By Proposition~\ref{PropEqui2} and  $x^*=4.77631$ given, the equilibrium point of the model (\ref{MatModelBio_gral}) is
\begin{equation}\label{Equil_Point_numeric}
x^*=4.77631,\ s^*=1.9427\ \text{and}\ u^*=0.148465.   
\end{equation}
Thus, the linearization of (\ref{MatModelBio_gral}) using (\ref{Matri_LinealControl}) is 
\begin{equation}\label{Matri_LinealControl_numerico}
A_0= \begin{pmatrix}
-0.3162 & 0\\
0.3580 & -0.0636
\end{pmatrix},\ A_1= \begin{pmatrix}
0 & -0.2734\\
0 & 0
\end{pmatrix}\ \text{and}\
B=\begin{pmatrix}
8.05728\\
0
\end{pmatrix}.    
\end{equation}
and the quasi-polynomial (\ref{quasi_control})  is
\begin{equation}\label{quasi_control_numerico}
q(\lambda,\tau,h)=\lambda^2+0.37832\lambda+0.020016+0.09791{\rm e}^{-7\lambda}-2.88459\,k_{r}\, {\rm e}^{-h\lambda}.
\end{equation}
Using Proposition~\ref{Prop_Regions}, the $\sigma$-stability regions in the parametric plane $h-k_r$ are given in Figure~\ref{SigmaRegions} for $\sigma>0$. Notice that these regions are concentric and their size decreases as $\sigma$ increases. Furthermore, a collapse of these regions to the point $(7.38,0.031)$ marked  with a red asterisk $(\textcolor{red}{\ast})$ can be observed when $\sigma=\sigma^{*}$, which implies that $\sigma^{*}=0.24$ is the maximum achievable exponential decay
of the linear systems (\ref{Matri_LinealControl_numerico}) or quasi-polynomial (\ref{quasi_control_numerico}).

\medskip
\noindent
The following observations may be proposed:
\begin{itemize}
	\item the mathematical model (\ref{MatModelBio_gral}) is unstable using the parameters (\ref{Parameters}) and $\tau=7$, as shown in Figure~\ref{Phase_sinControl_tau=7}.
	\item the Figure~\ref{SigmaRegions} shown geometrically all the gains (points) $h$ and $k_r$ 
	such that the delayed controller (\ref{ControlPR}) can ensure convergence to the equilibrium point $(s^*=1.9427,x^*=4.7763)$ when applied to the unstable model (\ref{MatModelBio_gral}) with (\ref{Parameters}) and $\tau=7$.
	\item a delayed controller of the form (\ref{ControlPR}) cannot assure convergence to the equilibrium point (\ref{Equil_Point_numeric}), when applied to the unstable model (\ref{MatModelBio_gral}) with (\ref{Parameters}) and $\tau=7$ if the gains $h$ and $k_r$ are outside these concentric regions of $\sigma$-stability.
	\item a delayed controller of the form (\ref{ControlPR}) with $h\in[0,2.6]$ cannot guarantee convergence to the equilibrium point (\ref{Equil_Point_numeric}) when applied to the unstable model (\ref{MatModelBio_gral}) with (\ref{Parameters}) and $\tau=7$.
	\item by the previous item, a controller with only proportional action $u(t)=k_px(t)$ (P control) cannot guarantee convergence to the equilibrium point (\ref{Equil_Point_numeric}) when applied to the unstable model (\ref{MatModelBio_gral}) with (\ref{Parameters}) and $\tau=7$.
\end{itemize}

\begin{figure}[ht]
	\centering
	\includegraphics[scale=.4]{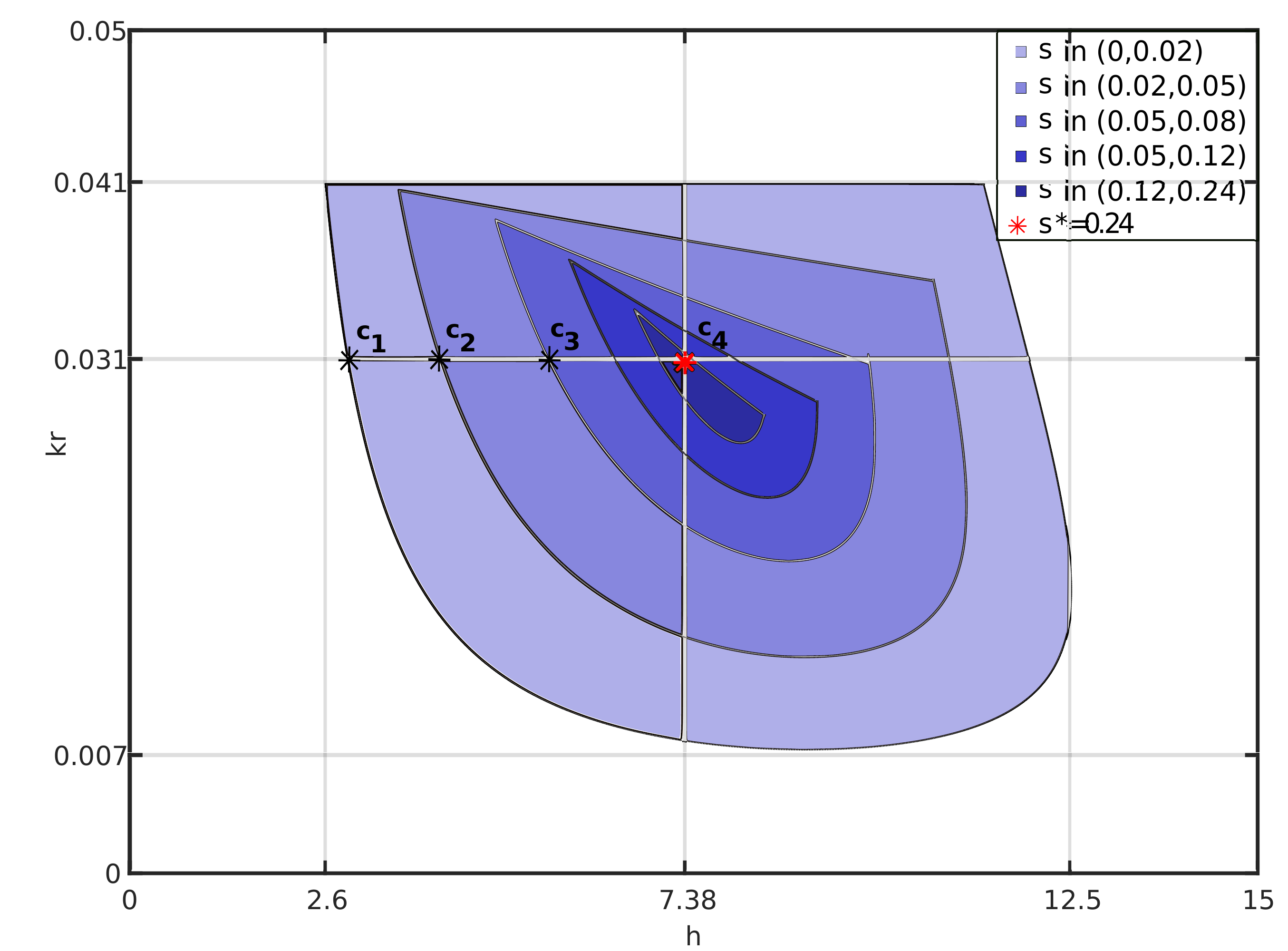}
	\caption{$\sigma$-stability regions to tune the delayed controller (\ref{ControlPR}).}\label{SigmaRegions}
\end{figure}

\medskip
To ratify the above, Figure~\ref{SystemsResp_Control} shows the response of the mathematical model (\ref{MatModelBio_gral}) when $u(t)=D$ and $u(t)=k_r x(t-h)$. Here, the initial condition is $(s_0=1.94,\, x_0=4.77)$ and the applied control is $u(t)=D=0.15$  for $t\in(0,T=500)$, once the instability of the model (\ref{MatModelBio_gral}) can be observed, then the control $u(t)=k_r x(t-h)$ is applied. In other words, consider the model given in (\ref{MatModelBio_gral}), the initial condition $(s_0=1.94,\, x_0=4.77)$, the parameters (\ref{Parameters}) and $\tau=7$, the mathematical model response (\ref{MatModelBio_gral}) is depicted in Figure~\ref{SystemsResp_Control}, when
\begin{equation*}
u(t)=\left\{  
\begin{array}{cc}
D, & t\in(0,T);  \\
k_r x(t-h), & t\in[T,t_f),
\end{array}
\right.    
\end{equation*}
where $T=500$, $t_f=1000$, $D=0.15$, and the gains $h$ and $k_r$ of the delayed controller are only the four points $c_j=(h_j,k_{r_j})$, $j=1,\dots,4$, marked with $(\ast)$ in Figure~\ref{SigmaRegions}, namely $c_1=(2.89,0.031)$, $c_2=(4.13,0.031)$, $c_3=(5.58,0.031)$ and $c_4=(7.38,0.031)$. Corresponding to border points of the $\sigma$-stable regions, when $\sigma=0,0.02,0.05$ and $0.24$. Furthermore, as mentioned in Definitions~\ref{def-sigma-est} and \ref{def-sigma-est-v2}, this $\sigma$ determines the exponential decay in the system response (\ref{MatModelBio_gral}), which can also be observed in Figure~\ref{SystemsResp_Control}. The signals delayed controller (\ref{ControlPR}) with $c_j=(h_j,k_{r_j})$, $j=1,\dots,4$ applied to the mathematical model (\ref{MatModelBio_gral}) as depicted in Figure~\ref{ControlSignal}. It is clear that this control signal is within the range established for $u(t)\in[0,\, 1]$ and this signal exponentially stabilize model fractional mathematical model (\ref{MatModelBio_gral}), even though it is initially unstable for $t\in(0,T)$. Calculation of the distribution of Hopf point contributed to enhancing the stability of the fermentation process, maintaining high product quality, and providing insights into system dynamics.

\begin{figure}[ht]
	\centering
	\includegraphics[scale=.4]{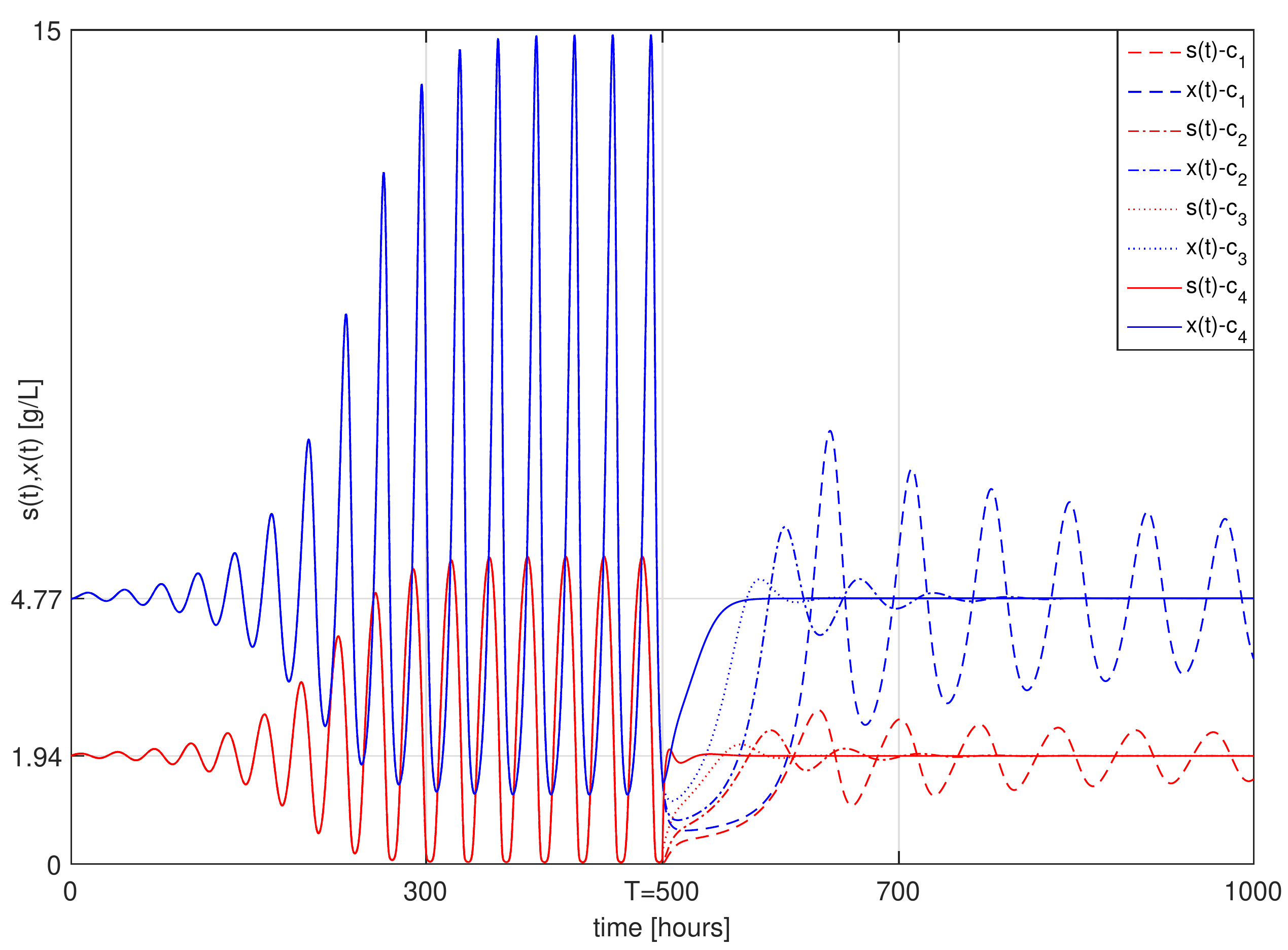}
	\caption{Mathematical model (\ref{MatModelBio_gral}) with delayed controller (\ref{ControlPR}) using $c_j=(h_j,k_{r_j})$, $j=1,\dots,4$. Here $c_1=(2.89,0.031)$, $c_2=(4.13,0.031)$, $c_3=(5.58,0.031)$ and $c_4=(7.38,0.031)$.}\label{SystemsResp_Control}
\end{figure}

\begin{figure}[H]
	\centering
	\includegraphics[scale=.4]{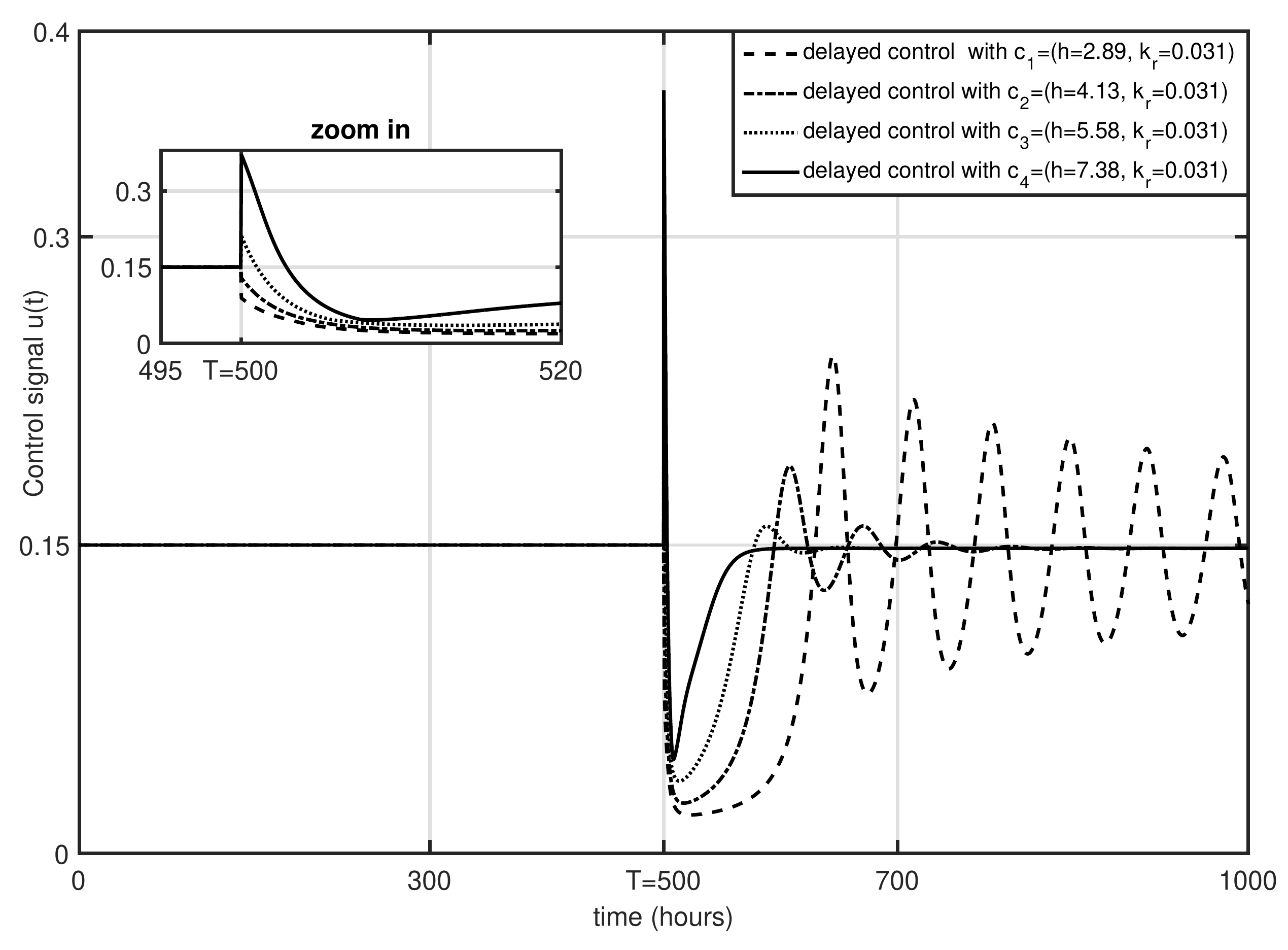}
	\caption{Signals delayed controller (\ref{ControlPR}) applied to the mathematical model (\ref{MatModelBio_gral}).}\label{ControlSignal}
\end{figure}

\section{Conclusions}

In this article  a fractional Lotka-Volterra mathematical model with time-delay is presented  fitting the data provided by a  biochemical process in a bioreactor  known as  {\it Zymomonas mobilis}. This model allows an analysis to determine Hopf bifurcations when the delay is used as a bifurcation parameter. Furthermore,  a well-founded proposal for the design and tuning of delayed controllers to stabilize the biochemical process is given. The efficiency and effectiveness of the theoretical results formulated are illustrated by  numerical simulation. It is worth pointing out those bifurcations regions of operation and control of oscillatory dynamics in the biosystem are of practical importance as they possess many potential applications in biofuels, medical science and biochemistry. This is due to the fact that these systems have a time delay in common in their signals as consequences of dead time in the quantification  of a state.

\section*{Appendix: Time-delay systems}\label{SubSec_PreliRes}

Since we consider delays in the mathematical model used to describe the dynamics of the biorector,  some concepts and criteria on the stability of time-delay systems are postulated in this Appendix.

\medskip
\noindent
Consider a time-delay nonlinear system of the form
\begin{align}\label{NonlinearSystem}
\dot{\vec{z}}(t)= G(\vec{z}(t),\vec{z}(t-\tau),\vec{u}(t)),
\end{align}
where $\vec{z}(t)=(z_1(t)\ z_2(t)\dots z_n(t))^{\intercal}$, 
$\vec{z}(t\!-\!\tau)=\left( z_1(t\!-\!\tau)\ z_2(t\!-\!\tau)\dots z_n(t\!-\!\tau)\right)^{\intercal}$, $\vec{u}(t)=(u_1(t)\ u_2(t)\dots u_m(t) )^{\intercal}$ 
and 
$G(\vec{z}(t),\vec{z}(t\!-\!\tau),\vec{u}(t))$=$(g_{1}(\vec{z}(t),\vec{z}(t\!-\!\tau),
\vec{u}(t))$  $g_{2}(\vec{z}(t),\vec{z}(t\!-\!\tau),\vec{u}(t))\dots g_{n}(\vec{z}(t),\vec{z}(t\!-\!\tau),\vec{u}(t)))^{\intercal}$. For $t\geq
0$ denote by $\vec{z}(t,\phi )$ the solution of the system with initial condition $%
\vec{\phi} \in \mathfrak{C}$, and by $\mathfrak{C}:=C([-\tau,0],\mathbb{R}^{n})$ the Banach space with norm $\Vert \vec{\phi} \Vert_{\tau}:=\max_{\theta \in \lbrack -\tau,0]}\Vert \vec{\phi} (\theta )\Vert $. Finally, $\Vert \cdot \Vert $ denotes the euclidean norm.

\medskip
\noindent
The equilibrium ${\vec{z}}^{ * }={ (z}_{ 1 }^{ * },{ z }_{ 2 }^{ * },\dots ,{ z }_{ n }^{ * })^{\intercal}$, $\vec{u}^*=(u_{1}^{*}\ u_{2}^{*}\dots u_{m}^{*})^{\intercal}$ is the one that satisfies
$G({ \vec{z}}^{ * },{ \vec{z}}_{ \tau  }^{ * },\vec{u}^*)=G({ \vec{ z }  }^{ * },{ \vec{ z }  }^{ * },\vec{u}^*)=0.$ Thus, the linearization of (\ref{NonlinearSystem}) at the equilibrium point is 
\begin{equation}\label{GralLineal}
\dot{\vec {z}}(t)=A_0\vec{z}(t) +A_1\vec{z}(t-\tau)+B\vec{u}(t),
\end{equation}
where
\begin{equation*}
\begin{array}{ll}
A_0=\left( \begin{matrix} 
\frac{{\partial} g_{ 1 }}{{ \partial} { z_1 }} & \frac{{\partial} g_{ 1 }}{{ \partial} { z_2 }} & \ldots & \frac{{\partial} g_{1}}{{ \partial} { z_n }}\\ 
\vdots  &  & \ddots  & \vdots  \\ 
\frac{{\partial} g_{ n }}{{ \partial} { z_1 }} & \frac{{\partial} g_{ n }}{{ \partial} { z_2 }} & \ldots & \frac{{\partial} g_{n}}{{ \partial} { z_n }}\\ 
\end{matrix} \right)\Big{|}_{
	\begin{array}{l}
	z=z^*\\
	u=u^*
	\end{array}
},
& A_1=\left( \begin{matrix} 
\frac{{\partial} g_{ 1 }}{{ \partial} { z_{1_{\tau}} }} & \frac{{\partial} g_{ 1 }}{{ \partial} { z_{2_{\tau}} }} & \ldots & \frac{{\partial} g_{1}}{{ \partial} { z_{n_{\tau}} }}\\ 
\vdots  &  & \ddots  & \vdots  \\ 
\frac{{\partial} g_{ n }}{{ \partial} { z_{1_{\tau}} }} & \frac{{\partial} g_{ n }}{{ \partial} { z_{2_{\tau}} }} & \ldots & \frac{{\partial} g_{n}}{{ \partial} { z_{n_{\tau}} }}\\ 
\end{matrix} \right) \Big{|}_{
	\begin{array}{l}
	z=z^*\\
	u=u^*
	\end{array}
}
\end{array}
\end{equation*}
$B= \begin{pmatrix}
\frac{{\partial} g_{ 1 }}{{ \partial} { u_1 }} & \frac{{\partial} g_{ 1 }}{{ \partial} { u_2 }} & \ldots & \frac{{\partial} g_{1}}{{ \partial} { u_m }}\\ 
\vdots  &  & \ddots  & \vdots  \\ 
\frac{{\partial} g_{ n }}{{ \partial} { u_1 }} & \frac{{\partial} g_{ n }}{{ \partial} { u_2 }} & \ldots & \frac{{\partial} g_{n}}{{ \partial} { u_m }}\\ 
\end{pmatrix} \Big{|}_{
	\begin{array}{l}
	x=x^*\\
	u=u^*
	\end{array}
}$, 

\noindent
with $\frac{{\partial} g_{ j }}{{ \partial} { z_k }}=\frac{\partial g_{j}(\vec{ z }(t), \vec{z}(t-\tau), \vec{u}(t))}{\partial z_k(t)}$,
$\frac{{\partial} g_{ j }}{{ \partial} { z_{k_{\tau}}}}=\frac{\partial g_{j}(\vec{z}(t), \vec{z}(t-\tau),\vec{u}(t))}{\partial z_k(t-\tau)}$ and \newline $\frac{{\partial} g_{ j }}{{ \partial} { u_l }}=\frac{\partial g_{j}(\vec{z}(t), \vec{z}(t-\tau),\vec{u}(t))}{\partial u_l(t)}$, $j,k=1,2,\dots,n$, $l=1,2,\dots,m$.
Now, consider a delayed controller of the form
\begin{equation}\label{GralControl}
\vec{u}(t)=K_1\vec{z}(t)+K_2\vec{z}(t-h),
\end{equation}
where $K_1,\ K_2\in\erre^{m\times n}$ and $h\in\erre^{+}$ is a delay (equal or different from $\tau$). If $h\neq\tau$, then  the closed-loop (\ref{GralLineal})-(\ref{GralControl}) is
\begin{equation}\label{GralLazoCerrado}
\dot{\vec{z}}(t)=\left[A_0+BK_1\right] \vec{z}(t) + A_1\vec { z}(t-\tau)+BK_2\vec{z}(t-h).
\end{equation}
Thus, the characteristic equation (quasi-polynomial) of system (\ref{GralLazoCerrado}) is of the form
\begin{equation}\label{GralQuasi}
q(\lambda,\tau)=\text{det}\{\lambda I_n -\left[A_0+BK_1\right]
-A_1{\rm e^{-\tau\lambda}}- BK_2{\rm e^{-h\lambda}}\}.
\end{equation}
\begin{definition}
	\cite{gu2003stability} \label{def-sigma-est} The systems (\ref{GralLazoCerrado}) is said $\sigma $-stable if the system response $z(t,\phi )$ satisfies the following inequality
	\begin{equation*}
	\Vert \vec{z}(t,\phi )\Vert\leq Le^{-\sigma t}\Vert \vec{\phi} \Vert _{\tau},\qquad t\geq 0,  
	\end{equation*}%
	where $L>0$, $\sigma \geq 0$, $\vec{\phi}:[-\tau,0]\rightarrow C([-\tau,0],\mathbb{R}^{n})
	$ is the initial condition. 
\end{definition}

\begin{definition}\cite{gu2003stability}\label{def-sigma-est-v2} 
	Consider the quasi-polynomial (\ref{GralQuasi}), $\sigma \in
	\mathbb{R}$ a positive constant and $$\lambda_{0}=\max_{j=1,\ldots ,\infty
	}\left\{ \mathrm{Re}\{\lambda_{j}\}\ |\ q(\lambda_{j},\tau)=0,\
	\lambda_{j}\in \mathbb{C}\right\},$$ where $\mathrm{Re}\{\lambda_{j}\}$ denote the real
	part of $\lambda_{j}$. Then, the system (\ref{GralLazoCerrado}) is said $\sigma$-stable if $\lambda_{0}\leq -\sigma$.
\end{definition}

%


\bibliographystyle{plain}
\bibliography{bioreactor4arxiv}

\end{document}